\documentclass[12pt,reqno]{amsart}
\usepackage{amscd,amssymb,amsmath,color}
  \usepackage[all]{xy}
\usepackage{url}
\numberwithin{equation}{section}

\input{prepictex.tex}
\input{pictex.tex}
\input{postpictex.tex}

\def\C{\mathbb C}
\def\N{\mathbb N}

\def\T{\mathbb T}
\def\Z{\mathbb Z}

\def\K{\mathcal K}

\def\aut{\operatorname{Aut}}

\def\id{\operatorname{id}}

\def\prim{\operatorname{Prim}}

\def\pol{\mathcal{O}}
\def\lra{\longrightarrow}
\def\ra{\rightarrow}

\def\sw#1{{\sb{(#1)}}}

\def\suc#1{{\sp{(#1)}}} 
\def\sut#1{{\sp{\langle #1\rangle}}} 
\def\tens{\mathop{\otimes}} 
\def\ot{\mathop{\otimes}}
\def\haar{\mathfrak{hm}}
\def\eps{\varepsilon}

\def\can{\mathsf{can}}
\def\flip{\mathsf{flip}}

\newtheorem{theo}{Theorem}[section]

\newtheorem{lemm}[theo]{Lemma}
\newtheorem{prop}[theo]{Proposition}
\newtheorem{defi}[theo]{Definition}
\newtheorem{rema}[theo]{Remark}

\newcounter{zlist}

\newcounter{blist}
\newenvironment{blist}{\begin{list}{(\alph{blist})}{
\usecounter{blist}\leftmargin2.5em\labelwidth2em\labelsep0.5em
\topsep0.6ex
\parsep0.3ex plus0.2ex minus0.1ex}}{\end{list}}
 
\newcounter{rlist}

\title[Quantum Sphere Bundle]{The Quantum Flag Manifold $\mathbf{SU_q(3)/\T^2}$ as an Example of a Noncommutative Sphere Bundle}

 \author{Tomasz Brzezi\'nski}
 \address{ Department of Mathematics, Swansea University, 
Swansea University Bay Campus,
Fabian Way,
Swansea,
  Swansea SA1 8EN, U.K.\ 
\& Department of Mathematics, University of Bia\l ystok, K.\ Cio\l kowskiego 1M, 15--245 Bia\l ystok, Poland} 
  \email{T.Brzezinski@swansea.ac.uk}   
\author{Wojciech Szyma\'nski}
\address{Department of Mathematics and Computer Science, University of Southern Denmark, 
Campusvej 55, 5230 Odense M, Denmark} 
\email{szymanski@imada.sdu.dk} 

\thanks{The first named author would like to express his gratitude to the members of the Department of Mathematics and Computer Science, University of Southern Denmark in Odense for very warm hospitality. His research  is supported in part by the Polish National Science Centre grant 2016/21/B/ST1/02438.
The second named author was  supported by  the 
DFF-Research Project 2, `Automorphisms and invariants of operator algebras', Nr. 7014--00145B, 2017--2021, and 
 the Villum Fonden Research Project `Local and global structures of groups and their algebras' (2014--2018).}

\date\today

\begin{document}

\begin{abstract}
The quantum flag manifold ${SU_q(3)/\T^2}$ is interpreted as a noncommutative bundle over the quantum complex projective plane with the quantum or Podle\'s sphere as a fibre. A connection arising from the (associated) quantum principal $U_q(2)$-bundle is described. 
\end{abstract}

\maketitle

\addtocounter{section}{-1}

\section{Introduction}
It is fairly well understood what a noncommutative principal circle bundle is (albeit with a commutative fibre). On the purely algebraic level, at which noncommutative principal bundles are understood as faithfully flat Hopf-Galois extensions or principal comodule algebras \cite{brzma}, \cite{h}, \cite{brzha1}, \cite{brzha2}, \cite{HajKra:pie}, a noncommutative principal circle bundle is simply a strongly $\mathbb{Z}$-graded algebra \cite{NasVan:gra}; the appearance of the $\mathbb{Z}$-grading here accounts for the Pontryagin duality between the circle group and the group of integers, while the strength of the grading reflects the freeness of the circle action. On the $C^*$-algebra level, as indicated in \cite{BegBrz:lin} and fully explored in  \cite{akl}, a noncommutative principal circle bundle can be understood as the Pimsner algebra associated to a Morita self-equivalence module or a noncommutative line bundle. Alas, what should play the role of a noncommutative bundle with the quantum sphere \cite{Pod:sph} or, equivalently, the quantum complex projective line \cite{vs} as a fibre does not appear to be understood at all. In this paper we try to address this lack of  understanding by looking closely at the quantum deformation of one of the simplest examples of sphere bundles, that of the flag manifold ${SU(3)/\T^2}$ over the complex projective plane. While the quantum sphere $S_q^2 \cong \C P_q^1$, just as its classical counterpart, is not a quantum group it is a quantum homogeneous space of the quantum unitary group $U_q(2)$, and we explore this richness of its algebraic structure to associate it to the noncommutative principal $U_q(2)$ bundle $SU_q(3) \lra \C P^2_q$. The total space of this associated quantum sphere bundle is identified with the quantum flag manifold ${SU_q(3)/\T^2}$. Admittedly this example cannot be taken as a beacon leading to a general method of approaching quantum sphere bundles, indeed only in a very limited number of cases the homogeneous space nature of the quantum sphere can be explored, it is nonetheless indicative of what one might expect of noncommutative sphere bundles and how they can be employed to gain better insight into noncommutative spaces. 

The paper is organised as follows. In the next section we recall elements of the theory of noncommutative principal and associated bundles and connections on them. In particular we list necessary Hopf algebra preliminaries and outline the way in which an idempotent for the module of sections of an associated bundle stems from a strong connection on a principal comodule algebra. The case of principal comodule algebras that arise from Hopf algebra epimorphisms or quantum homogeneous spaces is dealt with in detail. Section~2 recalls the construction of the quantum flag manifold ${SU_q(3)/\T^2}$, which is based on the use of (the $C^*$-algebra of continuous functions on) the quantum unitary group $U_q(3)$ and the gauge action by the two-torus. The main results of the paper are contained in Sections 3 and 4. The former is devoted to careful description of the quantum flag manifold ${SU_q(3)/\T^2}$ as a quantum sphere or $\C P^1_q$-bundle over the complex quantum  projective plane $\C P^2_q$. This is achieved in several steps. First, by analysing the primitive ideal structure of the $C^*$-algebra  $C(SU_q(3)/\T^2)$ we observe that  it is not isomorphic to the tensor
product $C(\C P_q^1)\otimes C(\C P_q^2)$. Thus a more elaborate construction needs to be sought if the quantum sphere bundle nature of ${SU_q(3)/\T^2}$ is to be revealed. Second, we interpret the algebra of continuous functions on the quantum projective plane $C(\C P_q^2)$ as a specific subalgebra of $C(SU_q(3)/\T^2)$. Third, working on the level of coordinate algebras we show in Lemma~\ref{lem.u2-bundle} that $\pol(SU_q(3)$ is a principal $\pol(U_q(2)$-comodule algebra over $\pol(\C P_q^2)$. In terms of quantum spaces this means that there is a quantum principal bundle
\begin{equation}\label{u2.su3.cp2}
\xymatrix{
U_q(2) \ar[r] & SU_q(3) \ar[d] \\
& \C P_q^2.}
\end{equation}
The embedding of $U_q(2)$ into $SU_q(3)$ that leads to the above bundle is shown to be compatible with the embedding of the two-torus into $SU_q(3)$ that yields the quantum flag manifold (in terms of coordinate $*$-algebras this is a statement of compatibility of corresponding epimorphisms of Hopf $*$-algebras, Lemma~\ref{lem.su3.u2}). This compatibility is then crucial for establishing that the algebra of coordinate functions on the quantum sphere i.e.\ on the quantum homogeneous space of $U_q(2)$ arising from the two-torus inclusion, when associated as a fibre to bundle \eqref{u2.su3.cp2}, gives rise to the quantum sphere bundle over $\C P_q^2$, with the quantum flag manifold as the total space,
\begin{equation}\label{cp1.fm.cp2}
\xymatrix{
\C P_q^1 \ar[r] & SU_q(3)/\T^2 \ar[d] \\
& \C P_q^2.}
\end{equation} 
This is established in Theorem~\ref{prop.frame}, which is the main result of Section~3, and, indeed, the key result of the paper. In the subsequent Lemma~\ref{lem.fm.j} we identify precisely $\C P_q^1$ as a fibre of 
 this quantum sphere bundle. Section~3 is completed by explicit calculation of the connection in \eqref{cp1.fm.cp2} obtained from a strong connection in \eqref{u2.su3.cp2}.


\section{Noncommutative bundles and connections}

\subsection{Notation, conventions and Hopf algebra preliminaries}
All vector spaces, algebras etc.\ are over the field $\C$ of complex numbers. The unadorned tensor product is over $\C$. By an algebra we mean a unital, associative algebra over $\C$. In an algebra $A$, the multiplication map is denoted by $\mu_A$ (or by juxtaposition of elements) and the identity is denoted by $1_A$ (or simply 1). Most of the algebras discussed will be $*$-algebras. We recall that a left (resp.\ right) $A$-module is projective if the multiplication (action) map has a splitting that is an $A$-module homomorphism. A left $A$-module $M$ is {\em faithfully flat} if the exactness of any sequence of {\em right} $A$-modules is equivalent to the exactness of this sequences tensored (over $A$) with $M$. 

Hopf algebras are always assumed to have bijective antipodes. The comultiplication in a Hopf algebra $H$ is denoted by $\Delta_H$, the counit by $\eps_H$ and the antipode by $S_H$. Whenever needed, we will us the Sweedler notation for comultiplication, for all $h\in H$,
$$
\Delta_H(h) = \sum h\sw 1\ot h\sw 2, \qquad (\id \ot \Delta_H)\circ \Delta_H(h) = \sum h\sw 1\ot h\sw 2\ot h\sw 3,
$$
etc.  Hopf algebras of (noncommutative) coordinate functions on quantum groups are denoted as $\pol(G_q)$, where $G$ is a classical group, the subscript $q$ indicates that it is a quantum version of $G$ and $\pol$ indicates that it is a Hopf algebra deforming functions on $G$. In all such cases we simplify the notation and write $\Delta_{G_q}$ for $\Delta_{\pol(G_q)}$, etc.

If $H$ is a Hopf algebra, by a {\em right $H$-comodule} we mean a vector space $V$ together with a map $\varrho_V: V\to V\ot H$, known as {\em coaction}, such that
$$
(\varrho_V\ot\id)\circ \varrho_V = (\id\ot\Delta_H)\circ \varrho_V, \qquad (\id \ot \eps_H)\circ\varrho_V = \id.
$$
On elements of $V$, the right coaction is denoted in the Sweedler-esque way: $\varrho_V(v) = \sum v\sw 0\ot v\sw 1$. 

Similarly, a left $H$-comodule is a vector space $W$ with a (counital and coassociative) map $\lambda_W: W\to H\ot W$. Given a right $H$-comodule $V$ and a left $H$-comodule $W$, their {\em cotensor product} is the vector subspace $V\Box_H W$ of $V\ot W$ defined as
$$
V\Box_H W := \{\sum_i v_i\otimes w_i\;|\; \sum_i \varrho_V(v_i)\ot w_i = \sum_i v_i\otimes \lambda_W(w_i)\}.
$$
If $H=\pol(G_q)$ then we abbreviate $V\Box_{\pol(G_q)}W$ to $V\Box_{G_q}W$.

An algebra $A$ that is also a right comodule of a Hopf algebra $H$ is called a {\em right comodule algebra} if the coaction $\varrho_A$ is an algebra homomorphism, where $A\otimes H$ is equipped with the standard tensor product algebra structure,
$$
(a\otimes h)(a'\otimes h') = aa'\otimes hh', \qquad 1_{A\otimes H} = 1_A\ot 1_H.
$$
Since the coaction $\varrho_H$ is an algebra map, the subspace 
$$
B = A^{\mathrm{co} H} = \{b\in A,\;| \; \varrho_A(b) = b\ot 1_H\}
$$
is a subalgebra of $A$ known as the {\em coaction invariant subalgebra} or {\em coinvariant subalgebra}. By the same token, $\varrho_A$ is a homomorphism of left $B$ modules (where $A$ is a left $B$-module by the restriction of the multiplication map, and $B$ acts $A\otimes H$ by multiplication in $A$ and trivially on $H$). In many cases, we deal with comodule algebras which are (noncommutative) coordinate algebras $\pol(X_q)$ of quantum spaces $X_q$ on which there are actions of quantum groups $G_q$. Algebraically this means the coaction of $\pol(G_q)$ on $\pol(X_q)$. In such cases $\varrho_{\pol(X_q)}$ is abbreviated to $\varrho_{X_q}$.

All Hopf algebras of type $\pol(G_q)$ we deal with admit a normalised left integral or a {\em Haar measure}. Recall that a normalised left integral on a Hopf algebra $H$ is a linear map $\haar: H\to \C$ such that, for all $h\in H$,
\begin{equation}\label{haar}
\sum h\sw 1\,\haar(h\sw 2) = \haar(h) 1_H, \qquad \haar(1_H) =1.
\end{equation}
If $A$ is a comodule algebra of a Hopf algebra $H$ with a Haar measure $\haar$, then one can construct the {\em averaging map}
 \begin{equation}\label{average}
 E: A\longrightarrow A^{\mathrm{co} H}, \qquad E= (\id_A\ot \haar)\circ \varrho_A.
 \end{equation}
Note that since a Haar measure is normalised, if $b\in A^{\mathrm{co} H}$, $E(b) =b$, and since $\varrho_A$ is left-linear over $A^{\mathrm{co} H}$, the map $E$ is a left-linear retraction of the inclusion $A^{\mathrm{co} H}\subseteq A$.

\subsection{Principal comodule algebras, associated modules and connections} 
The systematic study of connections in quantum principal bundles has been initiated in \cite{brzma}. In order to capture more faithfully the geometric contents of connections in noncommutative geometry this theory was developed further by the introduction of {\em strong connections} in \cite{h}, \cite{dgh}.  Motivated by the need to interpret examples arising from quantum homogeneous spaces, most notably quantum spheres, the theory was extended beyond Hopf algebras in \cite{BrzMaj:coa}, \cite{BrzMaj:geo}. Finally this approach to noncommutative principal bundles was  formalised in terms of {\em principal comodule algebras} e.g.\ in \cite{HajKra:pie} and {\em principal coalgebra extensions} in \cite{brzha0}. The details of proofs of the results quoted in this section can be found in \cite[Sections~5~\&~6]{Brz:ToK}.

\begin{defi}\label{def.princ.com.alg}
Let $H$ be a Hopf algebra and let $A$ be a right $H$-comodule algebra that is faithfully flat as a left module over its coinvariants subalgebra $B:= A^{\mathrm{co}H}$. $A$ is called a {\em principal comodule algebra} if the following {\em canonical Galois map}
\begin{equation}\label{can} 
\can : A\ot_B A \longrightarrow A\ot H, \qquad a\ot a'\longmapsto (a\otimes 1)\varrho_A(a'),
\end{equation}
is bijective.
\end{defi}

A comodule algebra with bijective canonical map is known as a {\em Hopf-Galois extension} (of the subalgebra of its coinvariants), introduced as a non-commutative generalisation of the classical Galois theory in \cite{ChaSwe:Hop} and \cite{KreTak:Gal}. Thus a principal comodule algebra is the same as a faithfully flat Hopf-Galois extension (by a Hopf algebra with bijective antipode). The (algebro-) geometric meaning of faithfully flat Hopf-Galois extensions was first analysed in detail by Schneider in \cite{Sch:pri}. 

The geometric meaning to principal comodule algebras and a useful way of determining whether a comodule algebra is principal is provided by the following lemma.

\begin{lemm}\label{lem.strong}
A right $H$-comodule algebra $A$ is a principal comodule algebra if and only if there exists a map $\ell: H\to A\ot A$ such that
\begin{subequations}\label{str.con}
\begin{equation}\label{con.norm}
\ell (1_H) = 1_A\otimes 1_A \qquad \mbox{(normalisation)},
\end{equation}
\begin{equation}\label{con.split}
\mu_A \circ \ell = 1_A\circ \eps_H  \qquad \mbox{(splitting property)}, 
\end{equation}
\begin{equation}\label{con.right}
(\id \otimes \varrho_A) \circ \ell = (\ell\otimes \id)\circ\Delta_H \qquad \mbox{(right colinearity)}, 
\end{equation}
\begin{equation}\label{con.left}
(\lambda_A \otimes\id) \circ \ell = (\id\otimes \ell)\circ\Delta_H  \qquad \mbox{(left colinearity)}, 
\end{equation}
\end{subequations}
where $\varrho_A: A\to A\otimes H$ is the right $A$-coaction and $\lambda_A: A\to H\otimes A$ is the associated left $H$-coaction,
\begin{equation}\label{left.coa}
\lambda_A = (S^{-1}\otimes \id)\circ \flip \circ \varrho_A, \qquad a\mapsto \sum S^{-1}(a\sw 1)\otimes a\sw 0.
\end{equation}
\end{lemm}

A map $\ell$ satisfying properties \eqref{str.con} in Lemma~\ref{lem.strong} is known as a {\em strong connection} or, more precisely, as a {\em strong connection form}. It is worth pointing out that right and left colinearity properties imply that 
\begin{equation}\label{middle}
(\id_A \ot \mu_A\ot \id_A)\circ (\ell \ot \ell)\circ \Delta_H \left(H\right)\subseteq  A\otimes A^{\mathrm{co} H} \otimes A\ .
\end{equation}

Let $A$ be a principal $H$-comodule algebra $A$ and let $V$ be a left $H$-comodule. The {\em module associated to $A$ and $V$} is defined as 
$$
\Gamma(A,V) := A\Box_H V.
$$
Since the coaction $\varrho_A$ is left linear over the coinvariant subalgebra $A^{\mathrm{co} H}$, $\Gamma(A,V)$ is a left $A^{\mathrm{co} H}$-module. While $A$ is interpreted as a noncommutative principal bundle, $\Gamma(A,V)$ is a module of sections of the associated bundle.

\begin{lemm}\label{lem.proj}
Let $A$ be a principal $H$-comodule algebra, and let $B:=A^{\mathrm{co}H}$. Then, for any left $H$-comodule $V$, the associated left $B$-module $\Gamma(A,V)=A\Box_H V$ is projective. Furthermore, if $V$ is finite-dimensional, then $\Gamma(A,V)$ is a finitely generated module.

Explicitly, if $\ell: H\to A\ot A$ is a strong connection form, then the map
\begin{equation}\label{split}
\sigma: \Gamma(A,V)\lra B\ot \Gamma(A,V), \qquad \sum_ia^i\ot v^i\mapsto \sum a^i\sw 0\ell(a^i\sw 1)\ot v^i,
\end{equation}
is a left $B$-module splitting of the multiplication map $B\ot \Gamma(A,V)\to  \Gamma(A,V)$.
\end{lemm}

As is well known, \cite{cq}, a splitting $\sigma$ induces a connection or a covariant derivative with values in the bimodule of universal one-forms $\Omega^1 B$ (the kernel of the multiplication map $B\ot B\to B$), by
\begin{equation}\label{Levi.Civita}
\nabla : \Gamma(A,V)\lra \Omega^1B\ot_B  \Gamma(A,V)\subset B \ot \Gamma(A,V), \qquad x\mapsto 1\ot x - \sigma(x).
\end{equation}
The connection property means that $\nabla$ satisfies the Leibniz rule in the form
$$
\nabla(b s) = d(b)s + b\nabla(s), \qquad \forall b\in  B, s\in \Gamma(A,V),
$$
where 
\begin{equation}\label{univ.dif}
d: B\longrightarrow \Omega^1 B, \qquad b \longmapsto 1\otimes b - b\otimes 1,
\end{equation}
 is the universal exterior derivative.

\begin{rema}\label{rem.idem}
\rm
If $V$ is a finite-dimensional comodule and $H$ has a Haar measure, then the idempotent in the algebra of matrices with entries in $B$ that defines $\Gamma(A,V)$ can be given explicitly as follows (see \cite{BohBrz:rel} or \cite[Section~7]{Brz:ToK}). First, we choose a basis $\{v_i\}$ of $V$. This choice uniquely determines elements $e_{ij}\in H$ by $\lambda_V(v_i) = \sum_{j} e_{ij}\otimes v_j$, which have the {\em coidempotent property} in the sense that $\Delta_H(e_{ij}) = \sum_ke_{ik}\ot e_{kj}$, $\eps_H(e_{ij}) = \delta_{ij}$. Next, we compute the value of a strong connection $\ell:H\to A\ot A$ on the $e_{ij}$,
$$
\ell(e_{ij}) = \sum_\nu \ell(e_{ij})\sut 1_\nu \ot \ell(e_{ij})\sut 2_\nu,
$$
and choose a (finite) basis $\{x_a\}$ for the subspace of $A$ generated by the $\ell(e_{ij})\sut 1_\nu$. Then, using the dual basis $\{\xi_a\}$ to $\{x_a\}$, we define the maps $\ell_a = (\xi_a\ot \id_A)\circ \ell: H\to A$, and the required idempotent is
\begin{equation}\label{idempotent}
	\mathbf{Q}=(Q_{(i,a),(j,b)}):=E(\ell_a(e_{ij})x_b),
\end{equation}
	where
 $E$ is the averaging map \eqref{average}.
 \end{rema}

Prime examples of principal comodule algebras arise from Hopf epimorphisms or {\em quantum homogeneous spaces}. 
\begin{prop}\label{prop.qhs}
Let $A$ and $H$ be Hopf algebras and assume that there exists a Hopf algebra epimorphism $\pi: A\to H$ and that $H$ admits a Haar measure $\haar: H\to \C$. View $A$ as a right $H$-comodule algebra with the coaction
\begin{equation}\label{qhs.coac}
\varrho_A: A\lra A\ot H, \qquad \varrho_A = (\id \ot \pi)\circ \Delta_A,
\end{equation}
where $\Delta_A$ is the comultiplication of $A$. Then $(A,\varrho_A)$ is a principal $H$-comodule algebra.
\end{prop}

Rather than proving this proposition, which is a consequence of one of Schneider's theorems \cite{Sch:pri}, we indicate the form of a strong connection $\ell: H\to A\ot A$ in this case (see e.g.\ \cite[Theorem~4.4]{BrzMaj:geo}, \cite{BegBrz:exp}). Since $\pi$ is a $\C$-linear epimorphism it has a $\C$-linear section $j': H\to A$ such that $j'(1)=1$. Using the integral $\haar$ we define the map $j: H\to A$ by
\begin{equation}\label{inv.j}
j(h) = \sum \haar\left(h\sw 1 S\left(\pi\left(j'(h\sw 2)\sw 2\right)\right)\right)j'(h\sw 2)\sw 2\haar\left(\pi\left(j'(h\sw 2)\sw 3\right) S \left( h\sw 3 \right)\right) .
\end{equation}
Since $\haar$ is a left integral, $j$ is a left and right $H$-comodule homomorphism in the sense that the following diagrams
\begin{equation}\label{j.bicom}
\xymatrix{H\ot H \ar[dd]_{\id\ot j} && H \ar[ll]_-{\Delta_H}  \ar[rr]^-{\Delta_H} \ar[d]^j && H\ot H\ar[dd]^{j\ot \id}\\
&& A \ar[d]^{\Delta_A} &&\\
H\ot A  && A\ot A \ar[ll]^-{\pi\ot \id}   \ar[rr]_-{\id\ot \pi} && A\ot H}
\end{equation}
commute. Since $\haar(1)=1$ and $\pi$ is a coalgebra morphism, $j$ is still a normalised section of $\pi$, i.e.\
\begin{equation}\label{sec}
\pi\circ j = \id, \qquad j(1) =1.
\end{equation}
A strong connection form is given by
\begin{equation}\label{l.j}
\ell : H \lra A\ot A, \qquad \ell = (S\ot \id)\circ \Delta_A\circ j.
\end{equation}
As shown in \cite[Theorem~4.4]{BrzMaj:geo} any left-invariant strong connection form on $A$ is of the form \eqref{l.j}. If the chosen $j'$ already satisfies  \eqref{j.bicom} then $j=j'$.

If $A$ is not isomorphic to $A^{\mathrm{co} H}\ot H$ as algebras, none of the maps $j$ will be an algebra homomorphism. This does not preclude existence of elements $h,h'\in H$ such that 
\begin{equation}\label{j.mult}
j(hh') = j(h)j(h'). 
\end{equation}
On the contrary, it is often the case that $j$ can be constructed multiplicatively from a (small) number of generators of $H$. This translates to the connection form \eqref{l.j}. For any $h,h'$ for which \eqref{j.mult} holds, $\ell$ is obtained by the following `sandwiching' procedure. For any algebra $A$, define the {\em sandwich transformation},
\begin{equation}\label{sandwich}
\mathfrak{s}_A : A\ot A\ot A\ot A\lra A\ot A, \qquad a_1\ot a_2\ot a_3 \ot a_4\mapsto a_3a_1\ot a_2a_4.
\end{equation}
\begin{lemm}\label{lem.sandwich}
In the situation of Proposition~\ref{prop.qhs} and $j$ satisfying \eqref{j.bicom} and \eqref{sec}, for all $h,h'\in H$ such that \eqref{j.mult} holds, the strong connection \eqref{l.j} comes out as
\begin{equation}\label{l.sandwich}
\ell(hh') = \mathfrak{s}_A (\ell(h)\ot \ell(h')) .
\end{equation}
\end{lemm}
\begin{proof}
The statement follows immediately from the fact that $\Delta_A$ is an algebra while $S$ is an anti-algebra homomorphism.
\end{proof}


\section{The quantum flag manifold}

\subsection{The algebra of functions on the quantum $\mathbf{SU(3)}$ group}

For $q\in(0,1)$, the $C^*$-algebra $C(SU_q(3))$ of `continuous
functions' on the quantum $SU(3)$ group is defined by Woronowicz
\cite{w3,w2} as the universal $C^*$-algebra generated by elements
$\{u_{ij}:i,j=1,2,3\}$ such that the matrix $\mathbf{u} = (u_{ij})_{i,j=1}^3$ is unitary and
\begin{equation}\label{determinant}
\sum_{i_1=1}^3\sum_{i_2=1}^3\sum_{i_3=1}^3 E_{i_1i_2i_3}
u_{j_1 i_1}u_{j_2 i_2}u_{j_3 i_3}=E_{j_1 j_2 j_3},
\;\;\;\;\; \forall(j_1,j_2,j_3)\in\{1,2,3\},
\end{equation}
where
\begin{equation}\label{inversions}
E_{i_1 i_2 i_3}=\begin{cases} (-q)^{I(i_1,i_2,i_3)} & \text{if}
\; i_r\neq i_s \; \text{for} \; r\neq s, \\ 0 & \text{otherwise,}
\end{cases}
\end{equation}
and $I(i_1,i_2,i_3)$ denotes the number of inversed pairs in the
sequence $i_1,i_2,i_3$. As pointed out by Br\c{a}giel \cite{bragiel},
$\{u_{ij}\}$ are coordinate functions of a quantum matrix
\cite{d,so,frt}. That is, the following relations are also satisfied
\begin{subequations}\label{qmatrix}
\begin{eqnarray}
u_{ij}u_{ik} & = & qu_{ik}u_{ij}, \;\;\;\;\; j<k, \label{qmatrix1} \\
u_{ji}u_{ki} & = & qu_{ki}u_{ji}, \;\;\;\;\; j<k, \label{qmatrix2} \\
u_{ij}u_{km} & = & u_{km}u_{ij}, \;\;\;\;\; i<k, \; j>m, \label{qmatrix3} \\
u_{ij}u_{km}-u_{km}u_{ij} & = & (q-q^{-1})u_{im}u_{kj},
\;\;\;\;\; i<k, \; j<m, \label{qmatrix4}
\end{eqnarray}
\end{subequations}
with $i,j,k,m\in\{1,2,3\}$. 
 The comultiplication 
 $$
 \Delta_{SU_q(3)}:C(SU_q(3))
\lra C(SU_q(3))\otimes C(SU_q(3)),
$$
 is a unital $C^*$-algebra homomorphism
such that
$$ \Delta_{SU_q(3)}(u_{ij})=\sum_{k=1}^n u_{ik}\otimes u_{kj}. $$

We denote by $\pol(SU_q(3))$ the subalgebra of $C(SU_q(3))$ generated
by $\{u_{ij}:i,j=1,2,3\}$. Since $u_{ij}^*=S(u_{ji})$, where
$S$ is the antipode (co-inverse) of $SU_q(3)$ \cite{frt,ks}, it follows
that $\pol(SU_q(3))$ is closed under the $*$-operation. 
Explicitly,
\begin{equation}\label{u*}
u_{ij}^* = (-q)^{j-i}\left(u_{i_1j_1}u_{i_2j_2} - q u_{i_1j_2}u_{i_1j_1}\right),
\end{equation}
where $i_1<i_2\in \{1,2,3\}\setminus \{i\}$ and $j_1<j_2\in \{1,2,3\}\setminus \{j\}$. 
Thus $\pol(SU_q(3))$,
the polynomial or coordinate algebra of $SU_q(3)$, is a dense $*$-subalgebra of $C(SU_q(3))$.

\subsection{The gauge action and its fixed point algebra}
The representation theory of the $C^*$-algebra $C(SU_q(3))$ has been described in detail by Br\c{a}giel in \cite{bragiel}. Six families of
irreducible representations, each indexed by elements $(\phi,\psi)$
of the $2$-torus have beed identified therein; see \cite{BrzSz:BiaProc} for a compact review. Among these, the $1$-dimensional irreducible representations $\pi_0^{\phi,\psi}$
of $C(SU_q(3))$ produce a surjective morphism of compact quantum groups
$\hat{\pi}_0:C(SU_q(3))\lra C(\T^2)$ (the diagonal embedding of $\T^2$ into $SU_q(3)$),
which gives rise to a gauge co-action of coordinate algebras 
\begin{equation}\label{hat.gamma.0}
\hat{\mu}:\pol(SU_q(3))\ra \pol(SU_q(3))\otimes
\pol(\T^2), \qquad \hat{\mu}=(\id\otimes\hat{\pi}_0)\circ \Delta_{SU_q(3)}.
\end{equation} 
Explicitly on the polynomial algebra $\pol(SU_q(3))$, $\hat{\pi}_0$ is a Hopf $*$-algebra epimorphism,
\begin{equation}\label{hat.pi.0}
\hat{\pi}_0:\pol(SU_q(3))\lra \pol(\T^2), \qquad \mathbf{u}\mapsto \begin{pmatrix}U_1 & 0 & 0\cr 0 & U_2 & 0\cr 0 & 0 & U_1^*U_2^*\end{pmatrix},
\end{equation}
where $U_1, U_2$ are unitary, group-like generators of the Hopf algebra $\pol(\T^2)$ of polynomial on $\T^2$ (the algebra of Laurent polynomials in two indeterminates). Hence the coaction comes out as
\begin{equation}\label{gamma.hat}
\hat{\mu}:\pol(SU_q(3))\ra \pol(SU_q(3))\otimes
\pol(\T^2), \qquad u_{ij}\mapsto \begin{cases} u_{ij}\tens U_j & \text{if} \; j=1,2, \\
u_{ij} \tens (U_1U_2)^{-1}  & \text{if} \; j=3. \end{cases}
\end{equation}

Equivalently,
$\mu:\T^2\lra\aut(C(SU_q(3)))$, $z\longmapsto \mu_z$,  is given by
\begin{equation}\label{gaugeaction}
\mu_z(u_{ij})=\begin{cases} z_j u_{ij} & \text{if} \; j=1,2, \\
(z_1z_2)^{-1}u_{ij} & \text{if} \; j=3. \end{cases}
\end{equation}
Here $z=(z_1,z_2)\in\T^2$ and each $z_i$ is a complex number of modulus $1$.
Let $C(SU_q(3)/\T^2)$ be the fixed point algebra of this gauge action, and let
$\pol(SU_q(3)/\T^2)=\pol(SU_q(3))\cap C(SU_q(3)/\T^2)$ be its polynomial
$*$-subalgebra, i.e.\ the subalgebra of coinvariants of $\hat{\mu}$,
\begin{equation}\label{poly.flag}
\pol(SU_q(3)/\T^2)= \pol(SU_q(3))^{\mathrm{co}\pol(\T^2)} = \{f\in  \pol(SU_q(3))\; |\; \hat{\mu}(f) = f\tens 1\}.
\end{equation}

Integration with respect to the Haar measure over $\T^2$ gives rise to
a faithful conditional expectation $\Phi:C(SU_q(3))\ra C(SU_q(3)/\T^2)$, namely
\begin{equation}\label{condexp}
\Phi(x)=\int_{z\in\T^2}\mu_z(x)dz.
\end{equation}
If $w$ is a monomial in $\{u_{ij}\}$ then $\Phi(w)$ is either $0$ or $w$.
Thus we have $\Phi(\pol(SU_q(3)))=\pol(SU_q(3)/\T^2)$, and whence
$\pol(SU_q(3)/\T^2)$ is a dense $*$-subalgebra of $C(SU_q(3)/\T^2)$.

There is a third equivalent way of understanding the gauge action, 
which we will also use: $\pol(SU_q(3))$ is a $\Z^2$-graded algebra with the degrees of the generators given by
\begin{equation}\label{grad.su3}
\deg(u_{i1}) = (1,0), \quad \deg(u_{i2}) = (0,1), \quad \deg(u_{i3}) = (-1,-1), \qquad i=1,2,3.
\end{equation}
From this point of view, $\pol(SU_q(3)/\T^2)$ is the  $(0,0)$-degree part of $\pol(SU_q(3))$. In what follows, 
we set 
\begin{equation}\label{gen.flag}
w_{ijk}=u_{i1}u_{j2}u_{k3}. 
\end{equation}


\section{The Noncommutative Sphere Bundle}
\subsection{The nontriviality of the potential quantum sphere bundle}

The classical  flag manifold $SU(3)/\T^2$ has the structure of a fibre bundle
with the base space $\C P^2$ and fibre $\C P^1\cong S^2$. Thus, it
is natural to expect that the quantum  flag manifold $SU_q(3)/\T^2$ should
have an analogous structure of a noncommutative `fibre bundle'
\begin{equation}\label{bundle}
\xymatrix{
\C P_q^1 \ar[r] & SU_q(3)/\T^2 \ar[d] \\
& \C P_q^2.}
\end{equation}
To begin with, we observe that the noncommutative bundle in question
cannot be trivial already on the topological level. Indeed, $C(\C P_q^1)$
is isomorphic to the minimal unitization of the compacts and $C(\C P_q^2)$
has the structure of an essential extension (see \cite{vs,hs}),
\begin{equation}\label{cp2extension}
0 \lra \K \lra C(\C P_q^2) \lra C(\C P_q^1) \lra 0.
\end{equation}
Thus the dual of $C(\C P_q^1)\otimes C(\C P_q^2)$ is
isomorphic to the following ordered set:
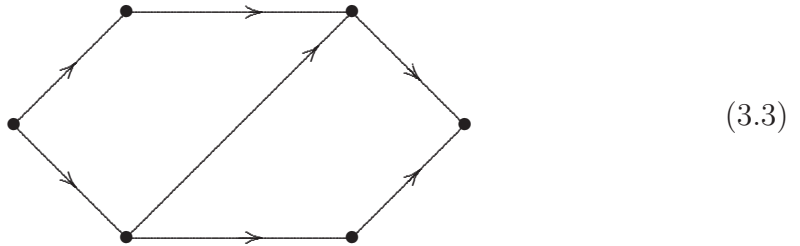
\begin{equation}
\beginpicture
\setcoordinatesystem units <1.5cm,1.5cm>
\setplotarea x from -2 to 2, y from -1 to 1

\put{$\bullet$} at 0 1
\put{$\bullet$} at 2 1
\put{$\bullet$} at -1 0
\put{$\bullet$} at 3 0
\put{$\bullet$} at 0 -1
\put{$\bullet$} at 2 -1

\setlinear
\plot 0 1 2 1 /
\plot 0 -1 2 -1 /
\plot -1 0 0 -1 /
\plot -1 0 0 1 /
\plot 0 -1 2 1 /
\plot 2 1 3 0 /
\plot 2 -1 3 0 /

\arrow <0.215cm> [.25,.75] from 1.1 1 to 1.2 1
\arrow <0.215cm> [.25,.75] from 1 -1 to 1.2 -1
\arrow <0.215cm> [.25,.75] from -.5 .5 to -.45 .55
\arrow <0.215cm> [.25,.75] from -.5 -.5 to -.45 -.55
\arrow <0.215cm> [.25,.75] from 1.65 .65 to 1.7 .7
\arrow <0.215cm> [.25,.75] from 2.5 .5 to 2.6 .4
\arrow <0.215cm> [.25,.75] from 2.5 -.5 to 2.6 -.4

\endpicture
\end{equation}

The topology on the primitive ideal space of $C(SU_q(3)/\T^2)$, described in detail for this particular case in \cite{BrzSz:BiaProc} 
and treated from a much more general perspective in \cite{StoDij:fla} and \cite{NevTus:hom}, corresponds to the following order (by inclusion):
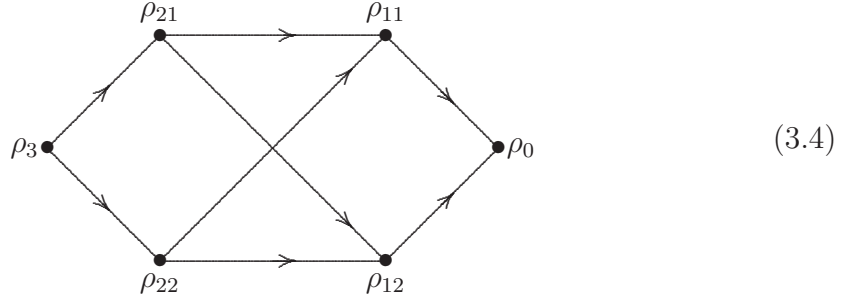
\begin{equation}\label{core}
\beginpicture
\setcoordinatesystem units <1.5cm,1.5cm>
\setplotarea x from -2 to 2, y from -1 to 1

\put{$\bullet$} at 0 1
\put{$\bullet$} at 2 1
\put{$\bullet$} at -1 0
\put{$\bullet$} at 3 0
\put{$\bullet$} at 0 -1
\put{$\bullet$} at 2 -1

\setlinear
\plot 0 1 2 1 /
\plot 0 -1 2 -1 /
\plot -1 0 0 -1 /
\plot -1 0 0 1 /
\plot 0 1 2 -1 /
\plot 0 -1 2 1 /
\plot 2 1 3 0 /
\plot 2 -1 3 0 /

\arrow <0.215cm> [.25,.75] from 1.1 1 to 1.2 1
\arrow <0.215cm> [.25,.75] from 1 -1 to 1.2 -1
\arrow <0.215cm> [.25,.75] from -.5 .5 to -.45 .55
\arrow <0.215cm> [.25,.75] from -.5 -.5 to -.45 -.55
\arrow <0.215cm> [.25,.75] from 1.65 .65 to 1.7 .7
\arrow <0.215cm> [.25,.75] from 1.65 -.65 to 1.7 -.7
\arrow <0.215cm> [.25,.75] from 2.5 .5 to 2.6 .4
\arrow <0.215cm> [.25,.75] from 2.5 -.5 to 2.6 -.4

\put{$\rho_3$} at -1.2 0
\put{$\rho_0$} at 3.2 0
\put{$\rho_{21}$} at 0 1.2
\put{$\rho_{11}$} at 2 1.2
\put{$\rho_{22}$} at 0 -1.2
\put{$\rho_{12}$} at 2 -1.2
\endpicture
\end{equation}

In view of \cite{BrzSz:BiaProc}, the Behncke and Leptin classification
of $C^*$-algebras with a finite dual \cite{bl} (or Elliott's
classification of $AF$-algebras \cite{e}) yields that 
for all $q\in(0,1)$, $C(SU_q(3)/\T^2)$ is isomorphic to the unique
unital $C^*$-algebra which admits a one-dimensional irreducible
representation and whose dual is isomorphic (as ordered set) to (\ref{core}).
This leads to the following. 

\begin{prop}\label{triviality}
The $C^*$-algebra $C(SU_q(3)/\T^2)$ is not isomorphic to the tensor
product $C(\C P_q^1)\otimes C(\C P_q^2)$.
\end{prop}
\begin{proof}
The primitive ideal space of $C(\C P_q^1)\otimes
C(\C P_q^2)$ has six elements, just like the primitive ideal
space of $C(SU_q(3)/\T^2)$. However, the hull-kernel topologies
on these two spaces are distinct. Indeed, $\prim(C(\C P_q^1)
\otimes C(\C P_q^2))$ contains a point whose closure has exactly
three elements, while $\prim(C(SU_q(3)/\T^2))$ has no such point.
\end{proof}

The algebra $C(\C P_q^2)$ is a $C^*$-subalgebra of $C(SU_q(3)/\T^2)$ in a natural
way, as follows (cf. \cite{vs}). The $C^*$-subalgebra of
$C(SU_q(3))$ generated by the first column matrix elements of $\mathbf{u}$, i.e.\ 
$u_{11}$, $u_{21}$ and $u_{31}$, may be identified with the
$C^*$-algebra $C(S^5_q)$ of continuous functions on the quantum
$5$-sphere. This $C^*$-subalgebra is invariant under the gauge
action $\mu$ of $\T^2$ on $C(SU_q(3))$. When restricted to
 $C(S^5_q)$, $\mu$ reduces to the generator-rescaling
circle action $u_{j1}\mapsto zu_{j1}$, $z\in\T$, whose fixed
point algebra is $C(\C P_q^2)$ (cf. \cite{vs,hs}). Thus,
in the setting of the present article, we have
\begin{equation}\label{cp2}
C(\C P_q^2)=C(SU_q(3)/\T^2)\cap C^*(u_{11},u_{21},u_{31}).
\end{equation}

\subsection{The noncommutative principal $U_q(2)$-bundle} The strategy for constructing bundle \eqref{bundle} is to interpret it as a bundle associated to the principal homogenous space fibration
\begin{equation}\label{su3.homog}
\xymatrix{
U_q(2) \ar[r] & SU_q(3) \ar[d] \\
& \C P_q^2.}
\end{equation}
Recall from \cite{frt} or \cite{h} that $U_q(2)$ is a compact matrix quantum group with the $C^*$-algebra of continuous functions $C(U_q(2))$ generated densely by the generators $u,\alpha,\gamma$, organised  into a unitary matrix 
\begin{equation}\label{uq2}
\mathbf{v} =  \begin{pmatrix} u & 0 & 0 \cr
0 & \alpha & -q\gamma^*u^*\cr 
0 & \gamma & \alpha^*u^*
\end{pmatrix}.
\end{equation}
The generator $u$ is central, while
\begin{equation}\label{su2.1}
\alpha\gamma = q \gamma\alpha,  \qquad \gamma\gamma^* = \gamma^*\gamma .
\end{equation}
The unitarity of $\mathbf{v}$ implies that $u$ is unitary, while $\alpha$ and $\gamma$ satisfy the remaining $SU_q(2)$,
$q$-commutation rules
\begin{equation}\label{su2.2}
\alpha\gamma^* = q \gamma^*\alpha, \qquad \alpha^*\alpha + \gamma\gamma^* = 1, \qquad \alpha\alpha^* + q^2 \gamma\gamma^*=1;
\end{equation}
see \cite{w1}. The Hopf algebra structure is that of a unitary compact quantum group, i.e.\ 
\begin{equation}\label{Hopf.u2}
\Delta_{U_q(2)}(v_{ij}) = \sum_{k=1}^3 v_{ik}\otimes v_{kj}, \quad \epsilon(v_{ij}) = \delta_{ij}, \quad S(v_{ij}) = v_{ji}^*.
\end{equation}
 The Hopf $*$-algebra (of polynomials in $C(U_q(2))$) generated by $\alpha, \gamma$ subject to relations \eqref{su2.1}-\eqref{su2.2} and central, unitary $u$ is denoted by $\pol (U_q(2))$. Thus,    $C(U_q(2)) \cong C(SU_q(2)) \otimes C(\T)$ as algebras, but not as (algebras of functions on) quantum groups.

Similarly to $C(SU_q(3))$, the algebra  $C(U_q(2))$ admits the torus action $\nu: \T^2 \to \mathrm{Aut}(C(U_q(2)))$ (arising from the torus embedding $\T^2 \hookrightarrow U_q(2)$), which for all $z= (z_1,z_2)\in \T^2$ induces the automorphisms
\begin{equation}\label{act.u2}
\nu_z: C(U_q(2))\to C(U_q(2)), \qquad u\mapsto z_1u, \quad \alpha\mapsto z_2\alpha, \quad \gamma\mapsto z_2\gamma.
\end{equation}
This action is reflected by the $\Z^2$-grading of $\pol(U_q(2))$, compatible with the $*$-involution in the sense that if $x$ is a homogeneous element of degree $(m,n)$, then $\deg(x^*) = (-m,-n)$:
\begin{equation}\label{grading.u2}
\deg(u) =(1,0), \quad \deg(\alpha) = \deg(\gamma) = (0,1).
\end{equation}
Explicitly the existence of the torus embedding means the Hopf $*$-algebra epimorphism
\begin{equation}\label{hat.pi.1}
\hat{\pi}_1: \pol(U_q(2)) \lra \pol(\T^2), \qquad 
 \mathbf{v}\mapsto \begin{pmatrix}U_1 & 0 & 0\cr 0 & U_2 & 0\cr 0 & 0 & U_1^*U_2^*\end{pmatrix},
\end{equation}
where $U_1, U_2$ are unitary, group-like generators of the Hopf algebra $\pol(\T^2)$.  The coaction is 
\begin{equation}\label{coac.t2.u2}
\begin{aligned}
\hat{\nu} = (\id\tens \hat{\pi}_1)\circ \Delta_{U_q(2)} :\: & \pol(U_q(2))\lra \pol(U_q(2))\otimes
\pol(\T^2),  \\
& u\mapsto u\tens U_1, \quad \alpha\mapsto \alpha\tens  U_2, \quad \gamma\mapsto \gamma\tens  U_2.
\end{aligned}
\end{equation}

\begin{lemm}\label{lem.su3.u2}
The $*$-map
\begin{equation}\label{pi}
\pi : \pol (SU_q(3)) \lra \pol (U_q(2)), \qquad \mathbf{u} \mapsto  \mathbf{v},
\end{equation}
is an epimorphism of Hopf $*$-algebras compatible with the $\Z^2$-gradings \eqref{grad.su3} and \eqref{grading.u2} or, equivalently, the diagram
\begin{equation}\label{pis}
\xymatrix{\pol (SU_q(3))\ar[rr]^-\pi \ar[dr]_{\hat{\pi}_0} & & \pol (U_q(2)) \ar[dl]^{\hat{\pi}_1} \\
& \pol(\T^2) &}
\end{equation}
(where the maps $\hat{\pi}_0$, $\hat{\pi}_1$ are given by \eqref{hat.pi.0} and \eqref{hat.pi.1}, respectively) is commutative.
\end{lemm}
\begin{proof}
Since both $\pol (SU_q(3))$ and $\pol (U_q(2))$ are Hopf $*$-algebras of matrix type, $\pi$ is a $*$-coalgebra morphism, obviously commuting with the antipode. Degrees of the first two columns of $\mathbf{u}$ and $\mathbf{v}$ clearly agree, and so do the third ones by the compatibility of the $\Z^2$-grading of $\pol (U_q(2))$ with the $*$-involution. It remains to prove that $\pi$ extends to the whole of $\pol (SU_q(3))$ as an algebra map. This can be done by analysing relations \eqref{inversions} and \eqref{qmatrix} and confirming that they are satisfied with all the $u_{ij}$ replaced by the $v_{ij}$. The only non-trivial choices for $(i,j,k)$ in \eqref{qmatrix1} are $(2,2,3)$ and $(3,2,3)$ (the constraint that $j<k$ forces at least one of the $v_{ij}$ or $v_{ik}$ to be zero in other cases). In the first case:
$$
v_{22}v_{23} = \alpha (-q \gamma^*u^*) = q (-q \gamma^*u^*) \alpha = q v_{23}v_{22},
$$
by the centrality of $u^*$ and by the first of equations \eqref{su2.2}. In a similar way,
$$
v_{32}v_{33} = \gamma \alpha^*u^*  = q \alpha^*u^* \gamma = q v_{33}v_{32},
$$
by the centrality of  $u^*$ and by application of the $*$-involution to the first of equations \eqref{su2.2}. The only nontrivial choices for $(i,j,k)$ in \eqref{qmatrix2} are $(2,2,3)$ and $(3,2,3)$, and then
$$
v_{22}v_{32} = \alpha\gamma = q \gamma \alpha = q v_{32}v_{22}, \qquad v_{23}v_{33} = -q \gamma^*u^*\alpha^*u^* = q\alpha^*u^*(-q \gamma^*u^*\alpha^*u^*) = v_{33}v_{23}, 
$$
by the first of equations \eqref{su2.1} and its $*$-conjugation as well as by the centrality of $u^*$. There is only one situation in which the relation \eqref{qmatrix3} is non-trivial, and it follows by the second of equations \eqref{su2.1} and the centrality of $u^*$:
$$
v_{23}v_{32} = \gamma (-q \gamma^*u^*) = -q \gamma^*u^*\gamma = v_{32}v_{23}.
$$
The last of relations \eqref{qmatrix} is nontrivial only when $(i,j)=(1,1)$ and $(k,m)=(2,2), (2,3)$ or $(i,j) = (2,2)$ and $(k,m) = (3,3)$. In the first two cases the right-hand side of \eqref{qmatrix4} vanishes and the commutation rule of the left-hand side of  \eqref{qmatrix4} follows by the centrality of $u$. In the third case,
$$
\begin{aligned}
v_{22}v_{33} - v_{33}v_{22} &= \alpha \alpha^*u^* - \alpha^*u^*\alpha = (1-q^2)\gamma\gamma^*u^*\\
& = (q-q^{-1})  (-q \gamma^*u^*)\gamma = (q-q^{-1}) v_{23}v_{32},
\end{aligned}
$$
by the combination of the second and third equations in \eqref{su2.2} and the second of \eqref{su2.1} and the centrality of $u^*$. 

It remains to check the  relations \eqref{inversions}. Because of the distribution of zeros in $\mathbf{v}$ there are only few nontrivial choices for $(j_1,j_2,j_3)$. The case of $(1,2,3)$ and its cyclic permutations follows by the centrality and unitarity of $u$ and by the third of equations \eqref{su2.2}:
$$
E_{123} v_{11}v_{22} v_{33} - E_{132} v_{11}v_{32} v_{23} = u\alpha\alpha^*u^* - qu\gamma (-q \gamma^*u^*) = \alpha\alpha^* + q^2\gamma\gamma^* =1 = E_{123}.
$$
Similarly, the case of $(j_1,j_2,j_3)=(2,1,3)$ and its cyclic permutations follows by the centrality and unitarity of $u$ and by the second of equations \eqref{su2.2}. The remaining non-trivial possibilities for the choice of $(j_1,j_2,j_3)$ are $(1,2,2)$ and $(1,3,3)$ and their cyclic permutations. In the first case (and its cycle permutations) one easily computes
$$
E_{123} v_{11}v_{22} v_{32} - E_{132} v_{11}v_{32} v_{22} = u\alpha\gamma - qu\gamma\alpha = 0,
$$
by the first of equations \eqref{su2.1}. The remaining cases follow by the $*$-conjugation of the first of equations \eqref{su2.1}.
\end{proof}

\begin{lemm}\label{lem.u2-bundle}
The right $\pol(U_q(2))$-coaction
\begin{equation}\label{coact.u2.su3}
\varrho_{SU_q(3)}: \pol(SU_q(3))\lra \pol(SU_q(3))\otimes\pol(U_q(2)), \quad \varrho_{SU_q(3)} = (\id\otimes \pi)\circ \Delta_{SU_q(3)},
\end{equation}
makes $\pol(SU_q(3)$ into a principal comodule algebra over $\pol(\C P_q^2)$.
\end{lemm}
\begin{proof}
Since the Hopf algebra $\pol(U_q(2))$ has an invariant integral (Haar measure), $\pol(SU_q(3))$ is a principal comodule algebra; see Proposition~\ref{prop.qhs}. It remains only to identify the coinvariants of the coaction $\varrho_{SU_q(3)}$. To this end first note that $U_q(2)$ is a middle term of a short exact sequence of quantum groups with the dual sequence of Hopf $*$-algebras
\begin{equation}\label{ses}
\xymatrix{
\pol(U(1)) \ar@{^{(}->}[r]^-j &  \pol(U_q(2)) \ar@{->>}[r]^p  & \pol(SU_q(2))\\
u \ar@{|->}[r]& u
 &
 \\
& {\begin{pmatrix} u & 0 & 0 \cr
0 & \alpha & -q\gamma^*u^*\cr 
0 & \gamma & \alpha^*u^*
\end{pmatrix}}\ar@{|->}[r] & {\begin{pmatrix} 1 & 0 & 0 \cr
0 & \alpha & -q\gamma^*\cr 
0 & \gamma & \alpha^*
\end{pmatrix}} },
\end{equation}
where $\pol(U(1)) = \pol(\T^1)$ is the $*$-algebra generated by a unitary $u$. The exactness of the sequence \eqref{ses} means first that the coaction $\varrho_{SU_q(3)}$ restricts to the $j(\pol(U(1)))$-coaction on the coinvariants of the $\pol(SU_q(2))$-coaction $(\id\otimes p)\circ \varrho_{SU_q(3)}$, and second that coinvariants of $\varrho_{SU_q(3)}$ are precisely these elements which are coinvariant under both the $\pol(SU_q(2))$ and  $j(\pol(U(1)))$-coactions. The composite $p\circ \pi : \pol(SU_q(3)) \to \pol(SU_q(2))$ is a Hopf algebra projection  which allows one to view the quantum sphere $S^5_q$ as a homogenous space of $SU_q(3)$ and thus to identify $\pol(S^5_q)$, which as was explained earlier is generated (as a $*$-algebra) by the first column of $\mathbf{u}$, as the subalgebra of $\pol(SU_q(3))$ coinvariant under the coaction $(\id \otimes p\circ \pi)\circ \Delta_{SU_q(3)} = (\id \otimes p)\circ \varrho_{SU_q(3)}$. Under $\varrho_{SU_q(3)}$ the first column of $\mathbf{u}$ transforms as $u_{i1}\mapsto u_{i1}\otimes u$, which is precisely the dual of the gauge $\T^1$-action $u_{i1}\mapsto zu_{i1}$, the fixed points of which are (polynomials on) $\C P^2_q$.
\end{proof}

Lemma~\ref{lem.u2-bundle} means that there is a quantum principal bundle \eqref{su3.homog}, thus the first aim of the strategy leading to the quantum sphere bundle  \eqref{bundle} has been achieved. Before we proceed to complete the interpretation of the quantum flag manifold as  a quantum sphere bundle over the quantum projective plane $\C P^2_q$, we digress on $K_0(\C P^2_q)$. More precisely, we show how to recover generators for $K_0(C(\C P_q^2))$ in the form of projection matrices with entries in $\pol(\C P_q^2)$ obtained in  \cite{DAL2010}, \cite{DAnLan:ant} (see also \cite{DAL2013}), from noncommutative vector bundles associated to the principal comodule algebra in Lemma~\ref{lem.u2-bundle} or bundle \eqref{bundle}.  

Let $V\suc 1$ be a one-dimensional vector space with a basis $\zeta$. The left coaction 
$\lambda_{V\suc 1}: V\suc 1\to \pol(U_q(2))\ot V\suc 1$ is defined by setting
$$
\lambda_{V\suc 1}(\zeta) = u\ot \zeta.
$$
In terminology of Remark~\ref{rem.idem} there is one coidempotent element $e_{11} = u$. In view of \eqref{fm.l.u}, a basis $\{x_a\}$ in Remark~\ref{rem.idem} can be chosen as $\{u^*_{11}, u^*_{21}, u^*_{31}\}$. Consequently $\ell_a(u) = u_{a1}$, $a=1,2,3$, and since $u_{a1}u^*_{b1}\in \pol(\C P_q^2)$, the idempotent corresponding to $\Gamma(\pol(SU_q(3), V\suc 1)$ comes out as the following $3\times 3$-matrix
\begin{equation}\label{idem.1-dim}
\mathbf{Q}\suc 1 = \left(u_{a1}u^*_{b1}\right)_{a,b=1}^3.
\end{equation}

In a similar way, let $V\suc{-1}$ be a one-dimensional vector space with a basis $\zeta$ and  left coaction $\lambda_{V\suc{-1}}(\zeta)=u^*\ot \zeta$,
i.e.\ there is one coidempotent element $e_{11} = u^*$. In view of \eqref{fm.l.u}, a basis $\{x_a\}$ in Remark~\ref{rem.idem} can be chosen as $\{u_{11}, qu_{21}, q^2u_{31}\}$. Consequently $\ell_a(u) = q^{a-1}u_{a1}$, $a=1,2,3$, and since $u^*_{a1}u_{b1}\in \pol(\C P_q^2)$, the idempotent corresponding to $\Gamma(\pol(SU_q(3), V\suc{-1})$ comes out as the following $3\times 3$-matrix
\begin{equation}\label{idem.1-dim.2}
\mathbf{Q}\suc{-1} = \left(q^{a+b-2}\,u^*_{a1}u_{b1}\right)_{a,b=1}^3.
\end{equation}
As shown in \cite[Corollary~4.2]{DAnLan:ant} together with the class of the trivial bundle (free module), the classes of 
$\mathbf{Q}\suc{\pm 1}$ form a full set of generators of $K_0(C(\C P_q^2)) \cong \Z^3$.

\begin{rema}\rm 
Let $V\suc 2$ be a two-dimensional vector space with a basis $\zeta_2, \zeta_3$ and with the coaction 
$\lambda_{V\suc 2}: V\suc 2\to \pol(U_q(2))\ot V\suc 2$ defined as 
$$
\zeta_i = \sum_{j=2}^3 v_{ij}\tens \zeta_j.
$$
In view of \eqref{fm.l}, the subspace of $\pol(SU_q(3))$ generated by the first tensorands in $\ell(v_{ij})$, $i,j=2,3$ has a basis $u^*_{ai}$, $a=1,2,3$, $i=2,3$, and hence $\ell_{ai}(v_{kl}) = \delta_{ik}u_{al}$. The Haar measure on the basis of $\pol(U_q(2))$ is given by
$$
\haar\left(\left(\gamma\gamma^*\right)^n\right) = \frac{q^2 -1}{q^{2n+2} - 1}, \qquad n=0,1,\ldots
$$
and zero elsewhere (see, e.g.\ \cite{BrzSz:BiaProc}). The procedure described in Remark~\ref{rem.idem} gives the idempotent for $\Gamma(\pol(SU_q(3), V\suc 2)$ in terms of the following $12\times 12$-matrix
\begin{equation}\label{idem.2-dim}
\begin{aligned}
\mathbf{Q}\suc 2 &= (Q\suc 2_{(iak)(jbl)}), \quad a,b=1,2,3, \; i,j,k,l = 2,3, \\
Q^\suc 2_{(iak)(jbl)} &= \frac{q^{2(3-j)}}{1+q^2} \delta_{ik}\delta_{jl} \left(u_{a2}u^*_{b2} + u_{a3}u^*_{b3}\right).
\end{aligned}
\end{equation}
Note that the entries in $\mathbf{Q}\suc 2$ are non-zero only if $i=k$ and $j=l$, thus this idempotent is equivalent to the $6\times 6$-matrix labelled by $(ia)$, $(jb)$. This matrix comprises of two identical rescaled and diagonally placed $3\times 3$ idempotent blocks
\begin{equation}\label{idem.2-dim.red}
\mathbf{\bar{Q}}\suc 2 = \left(u_{a2}u^*_{b2} + u_{a3}u^*_{b3}\right)_{a,b=1}^3.
\end{equation}
Consequently, both $\mathbf{Q}\suc 2$ and $\mathbf{\bar{Q}}\suc 2$ belong to the same $K_0$-class. Note that since the matrix $\mathbf{u}$ is unitary,
$$
\mathbf{Q}\suc 1 + \mathbf{\bar{Q}}\suc 2 = \mathbf{1}.
$$
\end{rema}

\subsection{The quantum flag manifold as a quantum sphere bundle} The $(0,0)$-degree part of $\pol(U_q(2))$ is the unital $*$-subalgebra generated by 1, $\gamma\gamma^*$ and $\alpha\gamma^*$, and thus it is the polynomial algebra of the standard Podle\'s sphere \cite{Pod:sph} or the quantum complex projective line \cite{vs}, $\pol(\C P_q^1)$. Equivalently,
\begin{equation}\label{poly.cp1}
\pol(\C P_q^1) = \pol(U_q(2))^{\mathrm{co}\pol(\T^2)} = \{a\in \pol(U_q(2))\; |\; \hat{\nu}(a)= a\ot 1\},
\end{equation}
where the coaction $\hat{\nu}$ is given by \eqref{coac.t2.u2}.
Noting that
$$
\gamma\gamma^* = -q^{-1} uv_{23}v_{32}, \qquad \alpha\gamma^* = -q^{-1}uv_{22}v_{23}, 
$$
and that
$u v_{32}v_{33}, uv_{22}v_{33} \in \pol(\C P_q^1)$, and using the fact that $\pol(U_q(2))$ is a Hopf algebra of matrix type, one easily finds,
$$
\begin{aligned}
\Delta_{U_q(2)}(\gamma\gamma^*) & = -q^{-1}\sum_{i,j=2}^3 uv_{2i}v_{3j}\otimes uv_{i3}v_{j2}  \in \pol(U_q(2))\otimes \pol(\C P_q^1), \\
\Delta_{U_q(2)}(\alpha\gamma^*) & = -q^{-1}\sum_{i,j=2}^3 uv_{2i}v_{3j}\otimes uv_{i2}v_{j3}  \in \pol(U_q(2))\otimes \pol(\C P_q^1) .
\end{aligned}
$$
Therefore, $\pol(\C P_q^1)$ is a left coideal subalgebra of $\pol(U_q(2))$, i.e.\ the comultiplication $\Delta_{U_q(2)}$ of $\pol(U_q(2))$ restricts to a left coaction on $\pol(\C P_q^1)$, and we can consider the bundle with standard fibre $\C P_q^1$ associated to the quantum principal bundle \eqref{su3.homog}, the projective module of non-commutative sections of which is the cotensor product
$$
\pol(SU_q(3))\Box_{U_q(2)}\pol(\C P_q^1).
$$

\begin{rema}\rm
Consider a classical, compact fibre bundle 
$$
\xymatrix{K \ar[r] & M\ar[d] \\ &  B.}
$$
 Then the fibre $K$ can be recovered algebraically, via the Gelfand duality, from the projection 
$M \to B$ as follows.  $C(K)$ is the largest $C^*$-algebra with the property that there exists a surjective $*$-homomorphism $C(M) \to C(K)$ which  
sends all elements of $C(B)$ to scalar multiples of the identity. Lemma \ref{lem.su3.u2} immediately implies that this is exactly the case with the 
noncommutative bundle \eqref{bundle}. Namely, viewed on the corresponding $C^*$-algebra level, the restriction of $*$-homomorphism $\pi:
C(SU_q(3))\to C(U_q(2))$ to $C(SU_q(3)/\T^2)$ is a surjective map onto $C(\C P_q^1)$ which sends all elements of $C(\C P_q^2)$ to scalars. Furthermore, 
the $C^*$-algebra $C(\C P_q^1)$ satisfies the maximality assumption as well. 
\end{rema}

The main aim of this paper is achieved in the following theorem. 

\begin{theo}\label{prop.frame}
As left $\pol(\C P_q^2)$-modules
\begin{equation}\label{iso.frame}
\pol (SU_q(3)/\T^2)\cong \pol(SU_q(3))\Box_{U_q(2)}\pol(\C P_q^1).
\end{equation}
\end{theo}
\begin{proof}
Let 
$$
\pol(SU_q(3)) = \bigoplus_{(m,n)\in \Z^2} \pol(SU_q(3))_{(m,n)},
$$
be the $\Z^2$-grading decomposition of $\pol(SU_q(3))$, so that $\pol (SU_q(3)/\T^2) = \pol(SU_q(3))_{(0,0)}$. Note that,
$$
\pi\left(\pol(SU_q(3))_{(0,0)}\right) = \pol(\C P_q^1),
$$
where $\pi$ is the degree-preserving Hopf $*$-algebra epimorphism \eqref{pi}. Since on the matrix $\mathbf{u}$ of generators of $\pol(SU_q(3))$ the grading is defined column-wise, 
$$
\Delta_{SU_q(3)} \left(\pol(SU_q(3))_{(m,n)}\right) \subset \pol(SU_q(3))\otimes \pol(SU_q(3))_{(m,n)}.
$$
In particular,
$$
\begin{aligned}
\varrho_{SU_q(3)}\left(\pol(SU_q(3)/\T^2)\right)  \subset & \pol(SU_q(3))\otimes \pi\left(\pol(SU_q(3))_{(0,0)}\right)\\
& = \pol(SU_q(3))\otimes \pol(\C P_q^1).
\end{aligned}
$$
Furthermore, since $\pi$ is a coalgebra homomorphism 
$$
\varrho_{SU_q(3)}\left(\pol(SU_q(3)/\T^2)\right)  \subset \pol(SU_q(3))\Box_{U_q(2)}\pol(\C P_q^1).
$$
Being a (counital) coaction, $\varrho_{SU_q(3)}$ is an injective map, and since  $\pol(\C P_q^2)$ is the subalgebra of coinvariants,  $\varrho_{SU_q(3)}$ is a left $\pol(\C P_q^2)$-module monomorphism. Thus the restriction  
$$
\varphi := \varrho_{SU_q(3)}\mid_{\pol(SU_q(3)/\T^2)} : \pol(SU_q(3)/\T^2)\lra \pol(SU_q(3))\Box_{U_q(2)}\pol(\C P_q^1),
$$
is a left $\pol(\C P_q^2)$-module monomorphism. We need to show that it is an epimorphism too. Any element of $\pol(SU_q(3))\Box_{U_q(2)}\pol(\C P_q^1)$ is necessarily of the form $\varrho_{SU_q(3)}(a)$, for some $a\in \pol(SU_q(3))$. Indeed, $x\in \pol(SU_q(3))\tens\pol(\C P_q^1)\subset \pol(SU_q(3))\tens\pol(U_q(2))$ is an element of  $\pol(SU_q(3))\Box_{U_q(2)}\pol(\C P_q^1)$ if and only if 
$$
(\varrho_{SU_q(3)}\tens \id)(x) = (\id\tens \Delta_{U_q(2)})(x).
$$
Applying $\id\tens\id\tens\eps_{U_q(2)}$ to both sides of this equality, in view of the linearity of all the maps involved as well as counitality of comultiplications, we thus obtain that
$$
x = \varrho_{SU_q(3)}(a), \quad \mbox{where}\quad a=\left(\id\tens \eps_{U_q(2)}\right)(x) \in \pol(SU_q(3)),
$$
as claimed.
We need to prove that $a\in \pol(SU_q(3)/\T^2)$ or, equivalently, that $\hat{\mu}(a) = a\tens 1$, where $\hat{\mu}$ is the gauge coaction \eqref{hat.gamma.0}. Since $\varrho_{SU_q(3)}(a) \in \pol(SU_q(3))\tens \pol(\C P_q^1)$,
$$
\begin{aligned}
\varrho_{SU_q(3)}(a)\tens 1 & = (\id\ot\hat{\nu})\circ \varrho_{SU_q(3)}(a)\\
&= (\id \ot \id\ot\hat{\pi}_1)\circ (\id \ot \Delta_{SU_q(3)})\circ (\id\ot\pi)\circ \Delta_{SU_q(3)}(a)\\
& = (\id \ot \id\ot\hat{\pi}_1)\circ (\id\ot\pi\ot \pi) \circ (\id \ot \Delta_{SU_q(3)})\circ \Delta_{SU_q(3)}(a)\\
&= (\id\ot\pi\ot \hat{\pi}_0) \circ (\id \ot \Delta_{SU_q(3)})\circ \Delta_{SU_q(3)}(a),
\end{aligned}
$$
where the first equality follows by the definitions of $\hat{\mu}$ and $\varrho_{SU_q(3)}$, the second one is a consequence of the fact that $\pi$ is a coalgebra morphism, while the third equality follows by \eqref{pis} in Lemma~\ref{lem.su3.u2}. Applying $\id\ot \epsilon_{U_q(2)}\ot \id$ to the just derived equality, and using the fact that $\pi$ is a coalgebra homomorphism we obtain
$$
a\ot 1 = (\id\ot \hat{\pi}_0) \circ \Delta_{SU_q(3)}(a) = \hat{\mu}(a),
$$
so that, by \eqref{poly.flag}, $a\in \pol(SU_q(3)/\T^2)$, as required.
\end{proof}

In the following lemma,  the quantum projective line $\C P_q^1$ is identified as a fibre of the bundle \eqref{bundle}. 
In passing we construct a strong connection on the principal homogeneous bundle \eqref{su3.homog}. Since, as an algebra, $\pol(U_q(2))$  is isomorphic to the tensor product $\pol(SU_q(2))\ot\pol(\T)$, its basis is the combination of bases for $\pol(SU_q(2))$ and $\pol(\T)$. We find it convenient to choose the following basis for $\pol(U_q(2))$,
\begin{equation}\label{basis}
\mathcal{B}: \quad
\begin{aligned} a_{klmn} &:= u^k \alpha^l\gamma^m(-q\gamma^*u^*)^n  = v_{11}^k\, v_{22}^l\, v_{32}^m\,v_{23}^n\, , \qquad k\in \Z, l,m, n\in \N, \cr
b_{klmn} &:=  u^k   \gamma^l(-q\gamma^*u^*)^m (\alpha^*u^*)^n=  v_{11}^k v_{32}^lv_{23}^m v_{33}^n, \quad  \quad k\in \Z, l,m \in \N, n\in \Z_+. 
\end{aligned}
 \end{equation}
\begin{lemm}\label{lem.fm.j}
Let $j: \pol(U_q(2)) \lra \pol(SU_q(3))$ be a linear transformation defined on the elements of basis $\mathcal{B}$ by
\begin{equation}\label{fm.j}
j\left(a_{klmn}\right) =  u_{11}^k\, u_{22}^l\, u_{32}^m\,u_{23}^n, \qquad j\left(b_{klmn}\right) = u_{11}^k\, u_{32}^l\,u_{23}^m\, u_{33}^n,
\end{equation}
where, for a negative $k$, $u_{11}^k$ means $u_{11}^{*-k}$. Then:
\begin{blist}
 \item $j$ is a splitting of $\pi$ \eqref{pi} that is bicolinear in the sense of diagrams \eqref{j.bicom} and hence it gives rise to a strong connection $\ell:  \pol(U_q(2)) \lra \pol(SU_q(3))\ot \pol(SU_q(3))$ through \eqref{l.j}.
 \item The map $j$ is compatible with the $\Z$-gradings \eqref{grad.su3} and \eqref{grading.u2}. In particular
 \begin{equation}\label{s2.in.fm}
 j\left(\pol(\C P_q^1)\right)\subseteq \pol(SU_q(3)/\T^2 )\cong \pol(SU_q(3))\Box_{U_q(2)}\pol(\C P_q^1).
 \end{equation}
 \end{blist}
\end{lemm}
\begin{proof}
That $\pi\circ j = \id$ is clear from \eqref{basis}, \eqref{fm.j} and \eqref{pi}. Since the coproducts as well as $\pi$ are algebra maps sufices it to check the commutativity of diagrams \eqref{j.bicom} on generators of $\pol(U_q(2))$. First let us consider the case of $v_{ij}$, $i,j,=2,3$. Then, on one hand
$$
\xymatrix{
v_{ij} \ar@{|->}[rr]^-{\Delta_{U_q(2)}} && \sum_{k=2}^3 v_{ik}\ot v_{kj}  \ar@{|->}[rr]^-{j \ot\id} && \sum_{k=2}^3 u_{ik}\ot v_{kj},
}
$$
while on the other 
$$
\xymatrix{
v_{ij} \ar@{|->}[rr]^-{\Delta_{SU_q(3)}\circ j} && \sum_{k=1}^3 u_{ik}\ot u_{kj}  \ar@{|->}[rr]^-{\id \ot\pi} && \sum_{k=1}^3 u_{ik}\ot v_{kj}=\sum_{k=2}^3 u_{ik}\ot v_{kj},
}
$$
where we have noted that for $j=2,3$, $v_{1j}=0$. Similarly, 
$$
\xymatrix{
v_{11} \ar@{|->}[rr]^-{\Delta_{U_q(2)}} &&  v_{11}\ot v_{11}  \ar@{|->}[rr]^-{j \ot\id} &&  u_{11}\ot v_{11},
}
$$
and 
$$
\xymatrix{
v_{11} \ar@{|->}[rr]^-{\Delta_{SU_q(3)}\circ j} && \sum_{k=1}^3 u_{1k}\ot u_{k1}  \ar@{|->}[rr]^-{\id \ot\pi} && \sum_{k=1}^3 u_{1k}\ot v_{k1}= u_{11}\ot v_{11},
}
$$
since $v_{k1}\neq 0$ only when $k=1$. 
This proves the commutativity of the the right-hand rectangle in \eqref{j.bicom}. The commutativity of the second rectangle is proven in a symmetric way. In that way the truth of statement (a) is established. The statement (b) follows immediately from the definitions of respective gradings and $j$ and the basis $\mathcal{B}$.
\end{proof}

In view of the Sandwich Lemma~\ref{lem.sandwich} the form of the strong connection $\ell$ \eqref{l.j} arising from the map $j$ \eqref{fm.j} is fully determined by the values of $\ell$ on the generators $v_{ij}$. Since $\mathbf{u}$ is a unitary matrix these come out as
\begin{subequations}\label{fm.l}
\begin{equation}\label{fm.l.u}
\ell(u) = \sum_{k=1}^3 u_{k1}^*\ot u_{k1}, \qquad \ell(u^*) = \sum_{k=1}^3 Su_{1k}^*\ot u_{k1}^* = \sum_{k=1}^3 q^{2(k-1)}u_{k1}\ot u_{k1}^* , 
\end{equation}
\begin{equation}\label{fm.l.v}
\ell(v_{ij}) = \sum_{k=1}^3 u_{ki}^*\ot u_{kj}, \qquad i,j=2,3,
\end{equation}
\end{subequations}
where the third equality in \eqref{fm.l.u} follows by the definition of the antipode and the formulae \eqref{u*}.
In terms of the $\pol(SU_q(2))$ generators
\begin{equation}\label{fm.l.su2}
\begin{aligned}
\ell(\alpha) &= \sum_{k=1}^3 u_{k2}^*\ot u_{k2}, \qquad \ell\left(\alpha^*\right) = \sum_{k,l=1}^3 u_{k1}^*u_{l3}^*\ot u_{l3}u_{k1},\\
\ell(\gamma) &= \sum_{k=1}^3 u_{k3}^*\ot u_{k2}, \qquad \ell\left(\gamma^*\right) = -q^{-1} \sum_{k,l=1}^3 u_{k1}^*u_{l2}^*\ot u_{l3}u_{k1},
\end{aligned}
\end{equation}
where the Sandwich Lemma~\ref{lem.sandwich} has been used.

Combining the formula \eqref{split} with the isomorphism \eqref{iso.frame} given by the restriction of the $\pol(U_q(2))$-coaction  $\varrho_{SU_q(3)} = (\id\ot \pi)\circ \Delta_{SU_q(3)}$ on $\pol(SU_q(3))$ to the quantum flag manifold we obtain the following formulae for the splitting of the multiplication map
\begin{equation}\label{split.flag}
\begin{aligned}
\sigma: \pol(SU_q(3)/\T^2) & \lra \pol(\C P_q^2) \ot \pol(SU_q(3)/\T^2), \\
 a & \longmapsto \sum a\sw 1\ell(\pi(a\sw 2)) =   \sum a\sw 1 S\left(j(\pi(a\sw 2)) \sw 1\right) \ot j(\pi(a\sw 2)) \sw 2,
 \end{aligned}
\end{equation}
where the Sweedler notation refers to the coproduct $\Delta_{SU_q(3)}$.  On the generators $w_{ijk}$ \eqref{gen.flag} of $\pol(SU_q(3)/\T^2)$, the coaction $\varrho_{SU_q(3)}$ comes out as
\begin{equation}\label{w.coact}
\begin{aligned}
\varrho_{SU_q(3)} \left(w_{ijk}\right) &= w_{ijk} \ot 1 + u_{i1}u_{j2}u_{k2} \ot uv_{22}v_{23} \\
& + u_{i1} \left(u_{j3}u_{k2} + q u_{j2}u_{k3}\right)\ot u v_{32} v_{23}
 + u_{i1}u_{j3}u_{k3} \ot u v_{32} v_{33}\\
 & = w_{ijk} \ot 1 -q u_{i1}u_{j2}u_{k2} \ot \alpha \gamma^* \\
 & - qu_{i1} \left(u_{j3}u_{k2} + q u_{j2}u_{k3}\right)\ot \gamma\gamma^*
 + u_{i1}u_{j3}u_{k3} \ot \gamma\alpha^*.
\end{aligned}
\end{equation}
Combining \eqref{w.coact} with \eqref{fm.l} and \eqref{split.flag}, equipped with the Sandwich Lemma~\ref{lem.sandwich} and using the fact that $\sum_{m=1}^3 u_{km}u^*_{lm} = \delta_{kl}$ one finds that
\begin{equation}\label{split.flag.gen}
\begin{aligned}
\sigma\left(w_{ijk}\right) &= w_{ijk}\ot 1 + u_{i1}u_{j3}\sum_{m,n=1}^3 u_{m3}^*u_{n1}^* \ot w_{nmk}\\
& + u_{i1}\sum_{l,m,n=1}^3  \left( u_{j2}u_{k2}u_{l2}^*u_{m2}^* + q u_{j2}u_{k3}u_{l2}^*u_{m3}^*  -  u_{j3}u_{k1}u_{l1}^*u_{m3}^*  \right)u_{n1}^*\ot w_{nml},
\end{aligned}
\end{equation}
and hence the connection \eqref{Levi.Civita} is
\begin{equation}\label{conn.flag.gen}
\begin{aligned}
\nabla\left(w_{ijk}\right) &= d(w_{ijk}) - u_{i1}u_{j3}\sum_{m,n=1}^3 u_{m3}^*u_{n1}^* \ot w_{nmk}\\
& - u_{i1}\sum_{l,m,n=1}^3  \left( u_{j2}u_{k2}u_{l2}^*u_{m2}^* + q u_{j2}u_{k3}u_{l2}^*u_{m3}^*  -  u_{j3}u_{k1}u_{l1}^*u_{m3}^*  \right)u_{n1}^*\ot w_{nml},
\end{aligned}
\end{equation}
where $d$ is the universal exterior derivative \eqref{univ.dif}.


\end{document}